\documentclass[a4paper,12pt]{article}

\usepackage[utf8]{inputenc}
\usepackage[T1]{fontenc}
\usepackage{amsmath,amssymb,amsfonts,amsthm,mathrsfs}
\usepackage{geometry}
\usepackage{graphicx}
\usepackage{float}
\usepackage[english]{babel}
\usepackage{tikz}
\usetikzlibrary{arrows}
\usetikzlibrary{decorations.markings}

\newenvironment{resume}[1]{
	\begin{list}{}{
		\setlength{\leftmargin}{1cm}
		\setlength{\rightmargin}{1cm}
	}\item[]
	{\bf #1}
	}{\end{list}}

\theoremstyle{plain}
\newtheorem{df}{Definition}
\newtheorem{thm}{Theorem}
\newtheorem{lem}{Lemma}

\newtheorem{rmq}{Remark}

\newtheorem*{thm*}{Theorem}
\newcommand{\mot}{}
\newtheorem*{thmref_interne}{\mot{}}
\newenvironment{thmref}[2]{
	\renewcommand{\mot}{#1 #2}
	\begin{thmref_interne}}
	{\end{thmref_interne}
}

\newcommand\mb{\mathbb}
\newcommand\mc{\mathcal}

\newcommand\mr{\mathrm}
\newcommand\ms{\mathscr}

\newcommand\F{\mc{F}}
\newcommand\G{\mc{G}}
\newcommand\C[1]{\mb{C}^{#1}}

\newcommand\Df[1]{\mr{Diff}(\C{#1},0)}
\newcommand\Cg[1]{(\C{#1},0)}
\newcommand\Dfhat{\widehat{\mr{Diff}}(\mb{C},0)}

\title{
Formal classification of two-dimensional neighborhoods of genus $g\geq 2$ curves with trivial normal bundle
}

\author{Olivier Thom\footnote{The author is supported by CNRS, ANR-16-CE40-0008 Foliage.}}
\date{}

\begin{document}
\maketitle

\begin{resume}{Abstract:}
In this paper we study the formal classification of two-dimensional neighborhoods of genus $g\geq 2$ curves with trivial normal bundle.
We first construct formal foliations on such neighborhoods with holonomy vanishing along many loops, then give the formal/analytic classification of neighborhoods equipped with two foliations, and finally put this together to obtain a description of the space of neighborhoods up to formal equivalence.
\end{resume}

\section{Introduction}

\subsection{General setting}

Let $C$ be a complex curve of genus $g$.
We are interested in the different $2$-dimensional neighborhoods $S$ of $C$.
More precisely, two surfaces $S,S'$ equipped with embeddings $C\hookrightarrow S$, $C\hookrightarrow S'$ define formally/analytically equivalent neighborhoods if there exists neighborhoods $U,U'$ of $C$ in $S$ and $S'$ and a formal/analytic diffeomorphism $\varphi: U \rightarrow U'$ inducing the identity on $C$.
The equivalence of two neighborhoods is thus given by diagrams

\begin{center}
\begin{tikzpicture}
\draw (0,0) node[left] {$C$};
\draw [right hook-latex] (0,0) -- (2,0);
\draw (2,0) node[right] {$U\subset S$};
\draw [->] (-0.3,-0.3) -- (-0.3,-1.3);
\draw (0,-1.5) node[left] {$C$};
\draw [right hook-latex] (0,-1.5) -- (2,-1.5);
\draw (2,-1.5) node[right] {$U'\subset S'$};
\draw [->] (2.2,-0.3) -- (2.2,-1.3);
\draw (0.1,-0.8) node[] {$id$};
\draw (2.2,-0.8) node[right] {$\varphi$};
\end{tikzpicture}
\end{center}

We want to understand the classification of such neighborhoods up to equivalence.

The first invariants in this problem are the normal bundle $N_C$ of $C$ in $S$ and the self-intersection $C\cdot C = \mr{deg}(N_C)$ of the curve $C$.
If $C\cdot C<0$, Grauert's theorem (cf \cite{grauert} or \cite{cam}) tells that if the self-intersection is sufficiently negative (more precisely, if $C\cdot C< 2(2-2g)$), then $S$ is analytically equivalent to $N_C$ (ie. a neighborhood of $C$ in $S$ is analytically equivalent to a neighborhood of the zero section in the total space of $N_C$).

In the case $C\cdot C>0$, we can cite the works of Ilyashenko \cite{ilyashenko_imbedding_elliptic} on strictly positive neighborhood of elliptic curves and of Mishustin \cite{mishustin} for neighborhoods of genus $g\geq 2$ curves with large self-intersection ($C\cdot C > 2g-2$).
In both cases, the authors show that there is a huge family of non-equivalent neighborhoods (there are some functional invariants).

In the case $C\cdot C = 0$, the neighborhoods of elliptic curves have already been studied.
Arnol'd showed in \cite{arnold_bifurcations} that if $S$ is a neighborhood of an elliptic curve whose normal bundle $N_C$ is not torsion, $S$ is formally equivalent to $N_C$; if moreover $N_C$ satisfies some diophantine condition, then $S$ is analytically equivalent to $N_C$.
The case when $C$ is an elliptic curve and $N_C$ is torsion was studied in \cite{ltt}; in particular, it is shown that the formal moduli space (ie. with respect to formal classification) of such neighborhoods is a countable union of finite dimensional spaces.

The goal of this paper is to study the neighborhoods of genus $g\geq 2$ curves with trivial normal bundle under formal equivalence.

\subsection{Notations}

Throughout this paper, we will use the term $\Df{}$ to denote the group of germs of analytic diffeomorphisms of $\mb{C}$ at $0$; we will write $\widehat{\mr{Diff}}(\mb{C},0)$ the group of formal diffeomorphisms of $\mb{C}$ at $0$.

A \emph{formal neighborhood} $\widehat{S}$ of $C$ is a scheme $\ms{X}=(X,\widehat{\mc{O}})$ with $C$ as a subscheme such that there is an open covering $X=\cup U_i$ of $X$ with $\widehat{\mc{O}}\vert_{U_i} = (\mc{O}_C\vert_{U_i})[[y_i]]$, some coordinates $x_i$ on $C\cap U_i$ and some holomorphic functions $u_{ji}^{(k)}$ with $y_j = \sum_{k\geq 1} u_{ji}^{(k)}(x_i) y_i^k$ and $u_{ji}^{(1)}$ not vanishing on $U_i\cap U_j$.

If $S$ is an analytic neighborhood of $C$, then the completion $\widehat{\mc{O}}$ of $\mc{O}_S$ along $C$ is the structure sheaf of a formal neighborhood $\widehat{S}$ of $C$.
The natural inclusion $\mc{O}_S \hookrightarrow \widehat{\mc{O}}$ gives an injection $\widehat{S}\hookrightarrow S$ and allows us to see $S$ as a formal neighborhood.
We say that two analytic neighborhoods $S,S'$ are formally equivalent if $\widehat{S}$ and $\widehat{S'}$ are equivalent.

Let $S=\cup U_i$ be a covering of an analytic neighborhood $S$ and $(u_i,v_i)$ some analytic coordinates on $U_i$ with $C\cap U_i = \{v_i=0\}$.
A \emph{regular analytic foliation} on $S$ having $C$ as a leaf can be seen as a collection of submersive analytic functions $y_i=\sum_{k\geq 1}{y_i^{(k)}(u_i)v_i^k}$ on each $U_i$ such that there exist some diffeomorphisms $\varphi_{ji}\in\Df{}$ with $y_j = \varphi_{ji}\circ y_i$ on $U_i\cap U_j$.
In analogy, a \emph{regular formal foliation} on $S$ around $C$ (or on a formal neighborhood $\widehat{S}$ of $C$) is a collection of formal power series $y_i = \sum_{k\geq 1}{y_i^{(k)}(u_i)v_i^k}$ with $y_j=\varphi_{ji}\circ y_i$ for some $\varphi_{ji}\in \widehat{\mr{Diff}}(\mb{C},0)$ where the coefficients $y_i^{(k)}(u_i)$ are still analytic functions on $C\cap U_i$ and $y_i^{(1)}$ does not vanish on $C\cap U_i$ (otherwise stated, the divisor $\{y_i=0\}$ is equal to $\{v_i=0\}=C\cap U_i$).

\subsection{Results}

We will use the same strategy as in \cite{ltt}: first construct two "canonical" regular formal foliations $\F$, $\G$ on $S$ having $C$ as a leaf, then study the classification of formal/convergent bifoliated neighborhoods $(S,\F,\G)$, and finally put these together to obtain the formal classification of neighborhoods.

The first step, the construction of "canonical" foliations, is explained in section \ref{sec_construction}.
It has already been proved in \cite{clpt} that there exist formal regular foliations in $S$ having $C$ as a leaf.
Since we need to have some kind of unicity to be able to use these for the classification of neighborhoods, we will need to adapt the construction of \cite{clpt}.
The idea is to construct foliations whose holonomy is trivial along as many loops as possible.
For this, we fix a family $(\alpha_1,\ldots,\alpha_g,\beta_1,\ldots,\beta_g)$ of loops in $C$ which is a symplectic basis in homology and denote $A$-loops the loops $\alpha_i$ and $B$-loops the $\beta_i$.
We prove the following: 

\begin{thm}
\label{thm_constr_fol}
Let $C$ be a curve of genus $g\geq 2$ and $S$ a neighborhood of $C$ with trivial normal bundle.
Then there exists a unique regular formal foliation $\mc{F}$ on $S$ having $C$ as a leaf, such that the holonomy of $\mc{F}$ along $A$-loops is trivial.
\end{thm}



\begin{figure}[H]
\begin{center}
\begin{tikzpicture}
\draw[line width=0.4mm] (-4.2,0) -- (4.2,0);
\draw (4.2,0) node[right] {$C$};

\draw plot [domain=-4:4,smooth] (\x,{(\x+9)*(\x+9)/81});
\draw plot [domain=-4:4,smooth] (\x,{(\x+9)*(\x+9)/(2*81)});
\draw plot [domain=-4:4,smooth] (\x,{(\x+9)*(\x+9)/(4*81)});
\draw plot [domain=-4:0.5,smooth] (\x,{(\x+9)*(\x+9)/40});

\draw plot [domain=-4:4,smooth] (\x,{-(\x+9)*(\x+9)/81});
\draw plot [domain=-4:4,smooth] (\x,{-(\x+9)*(\x+9)/(2*81)});
\draw plot [domain=-4:4,smooth] (\x,{-(\x+9)*(\x+9)/(4*81)});
\draw plot [domain=-4:0.5,smooth] (\x,{-(\x+9)*(\x+9)/40});

\draw plot [domain=-4:4,smooth] (-\x,{(\x+9.5)*(\x+9.5)/81});
\draw plot [domain=-4:4,smooth] (-\x,{(\x+9.5)*(\x+9.5)/(2*81)});
\draw plot [domain=-4:4,smooth] (-\x,{(\x+9.5)*(\x+9.5)/(4*81)});
\draw plot [domain=-4:0,smooth] (-\x,{(\x+9.5)*(\x+9.5)/40});

\draw plot [domain=-4:4,smooth] (-\x,{-(\x+9.5)*(\x+9.5)/81});
\draw plot [domain=-4:4,smooth] (-\x,{-(\x+9.5)*(\x+9.5)/(2*81)});
\draw plot [domain=-4:4,smooth] (-\x,{-(\x+9.5)*(\x+9.5)/(4*81)});
\draw plot [domain=-4:0,smooth] (-\x,{-(\x+9.5)*(\x+9.5)/40});
\end{tikzpicture}
\end{center}
\caption{A bifoliated neighborhood of $C$}
\end{figure}
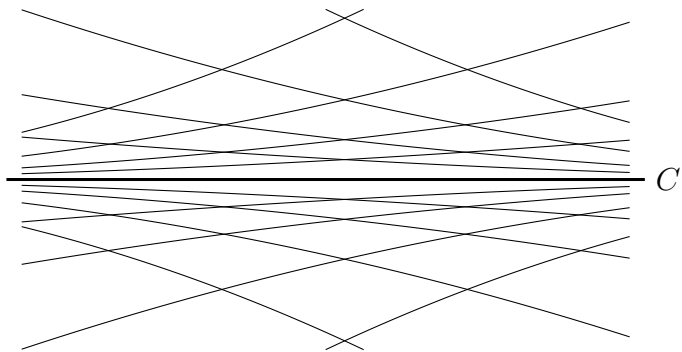

The second step, the classification of bifoliated neighborhoods, can be found in \cite{thom_these}.
We will explain in section \ref{sec_bif_classification} how this classification works in the generic case and show that a bifoliated neighborhood $(S,\F,\G)$ is characterised by the order of tangency $k$ between $\F$ and $\G$ along $C$, a $1$-form $\omega$ which controls how $\F$ and $\G$ differ at order $k+1$ and an additionnal invariant
\[
Inv(S,\F,\G)\in (\mr{Diff}(\mb{C},0))^{6g-3}/\sim
\]
(resp. $\widehat{Inv}(S,\F,\G)\in (\Dfhat)^{6g-3}/\sim$ for formal neighborhoods), where the relation $\sim$ is given by the action of $\mr{Diff}(\mb{C},0))$ on $(\Df{})^{6g-3}$ by conjugacy on each factor (resp. the action of $\Dfhat$ on $(\Dfhat)^{6g-3}$ by conjugacy on each factor).
This invariant is given by holonomies of the foliations $\F$ and $\G$ computed on a tangency curve $T_1$, ie. an irreducible component different from $C$ of the set of points at which $\F$ and $\G$ are tangent.

\begin{thmref}{Theorem}{\ref{thm_classif_bif}}
Let $C$ be a curve of genus $g\geq 2$. 
Let $(S,\F,\G)$ and $(S',\F',\G')$ be two bifoliated neighborhoods of $C$ with 
same tangency order $k$ and $1$-form $\omega$.
Suppose $k\geq 1$ and that $\omega$ has simple zeroes $p_1,\ldots,p_{2g-2}$.
Denote $T_1, T'_1$ the tangency curves passing through $p_1$ and compute the invariants $Inv(S,\F,\G)$ and $Inv(S',\F',\G')$ on the tangency curves $T_1, T'_1$.

Then $(S,\F,\G)$ and $(S',\F',\G')$ are analytically (resp. formally) diffeomorphic if and only if
\[
Inv(S',\F',\G')=Inv(S,\F,\G)
\]
(resp. $\widehat{Inv}(S',\F',\G') = \widehat{Inv}(S,\F,\G)$).
\end{thmref}


Moreover, we know which invariants come from a bifoliated neighborhood: if $((\varphi_i^1)_{i=1}^{2g},(\varphi_i^2)_{i=1}^{2g},(\varphi_j^3)_{j=2}^{2g-2})$ is a representant of $Inv(S,\F,\G)$, then the $\varphi_r^s$ must be tangent to identity at order $k$.
Moreover, if we write $\varphi_r^s(t) = t+a_r^st^{k+1} \text{ (mod $t^{k+2}$)}$, then the periods of $\omega$ must be $(a_i^2-a_i^1)_{i=1,\ldots,2g}$ (equation \eqref{eq_compatibility_1} in the text); $a_j^3$ must be equal to $\int_{p_1}^{p_j}\omega$ for $j=2,\ldots,2g-2$ (equation \eqref{eq_compatibility_2}); and the $(\varphi_i^s)_{i=1}^{2g}$ must be representations of the fundamental group of $C$ for $s=1,2$, ie. $[\varphi_1^s,\varphi_{1+g}^s]\ldots[\varphi_g^s,\varphi_{2g}^{s}] = id$ (equation \eqref{eq_representation}).

\begin{thmref}{Theorem}{\ref{thm_constr_bif}}
Let $((\varphi_i^1)_{i=1}^{2g},(\varphi_i^2)_{i=1}^{2g},(\varphi_j^3)_{j=2}^{2g-2})$ be some analytic/formal diffeomorphisms; let $k$ be an integer and $\omega$ a $1$-form.
They define a bifoliated analytic/formal neighborhood $(S,\F,\G)$ with $\F$ and $\G$ tangent at order $k$ and with $1$-form $\omega$ if and only if every $\varphi_r^s$ is tangent to identity at order (at least) $k$ and if they satisfy the relations \eqref{eq_compatibility_1}, \eqref{eq_compatibility_2} and \eqref{eq_representation}.
\end{thmref}

If the diffeomorphisms $\varphi_r^s$ are only formal, then the neighborhood is a priori only a formal neighborhood of $C$.
Note here that the relations \eqref{eq_compatibility_1} and \eqref{eq_compatibility_2} are in fact relations between jets of order $k+1$ of the $\varphi_r^s$, so the set of bifoliated neighborhoods modulo equivalence has huge dimension.
Indeed, the space of pairs of diffeomorphisms modulo common conjugacy is already infinite dimensional, even formally: if we fix one diffeomorphism $\varphi_1\neq id$ tangent to the identity, then the centralizer of $\varphi_1$ has dimension $1$ so that the set of pairs $(\varphi_1,\varphi_2)$ modulo common conjugacy has roughly speaking the same cardinality as the set of diffeomorphisms.

Finally, the last step (the formal classification of neighborhoods) is done in section \ref{sec_formal_classification}.
For the pair of canonical foliations constructed, the tangency order $k$ will be the Ueda index of the neighborhood (introduced by Ueda in \cite{ueda} and named by Neeman in \cite{neeman_ueda}), ie. the highest order such that there is a tangential fibration on $Spec(\mc{O}_S/I^{k+1})$ where $I$ is the ideal sheaf of $C$ in $S$.
Similarly, $\omega$ can be interpreted in term of the Ueda class of $S$.
We will define the space $\ms{V}(C,k,\omega)$ 
of neighborhoods with trivial normal bundle, fixed Ueda index equal to $k$ and fixed Ueda class given by $\omega$ in order to state the final theorem:

\begin{thmref}{Theorem}{\ref{thm_formal_classification}}
Let $C$ be a curve of genus $g\geq 2$, $1\leq k<\infty$ and $\omega$ a $1$-form on $C$ with simple zeroes.
Then there is an injective map 
\[
\Phi : \ms{V}(C,k,\omega) \hookrightarrow \Dfhat{}^g\times\Dfhat{}^g\times\Dfhat{}^{2g-3}/\sim
\]
where the equivalence relation $\sim$ is given by the action of $\Dfhat{}$ on $\Dfhat{}^N$ by conjugacy on each factor.

A tuple of diffeomorphisms $((\varphi_{i}^{(j)})_i)_{j=1}^3$ is in the image of $\Phi$ if and only if the $\varphi_{i}^{(j)}$ are tangent to the identity at order $k$ and if they satisfy the compatibility conditions \eqref{eq_compatibility_1} and \eqref{eq_compatibility_2}.
\end{thmref}

\section{Construction of foliations}
\label{sec_construction}

On the curve $C$ we can choose loops $\alpha_i,\beta_i$, $i=1,\ldots,g$ forming a symplectic basis of $H_1(C,\mb{C})$, ie. $\alpha_i\cdot \alpha_j = \beta_i\cdot \beta_j=0$ and $\alpha_i\cdot \beta_j = 1$ if $i=j$ and $0$ otherwise.
We call $A$-loops the loops $\alpha_i$ and $B$-loops the $\beta_i$.
Similarly, if $\omega$ is a $1$-form on $C$, we will call $A$-period (resp. $B$-period) of $\omega$ any integral $\int_{\alpha_i} \omega$ (resp. $\int_{\beta_i} \omega$).
\begin{figure}[H]
\includegraphics{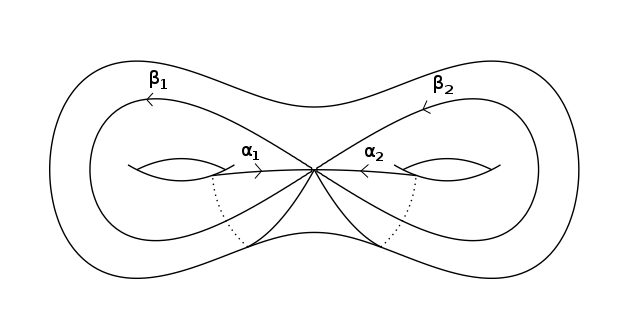}
\caption{$A$- and $B$-loops}
\end{figure}

\begin{df}
A foliation will be called $A$-canonical if its holonomy representation $\rho$ satisfies $\rho(\alpha_i) = id$ for all $i=1,\ldots,g$ and if the linear part of $\rho$ is trivial.
We define the notion of $B$-canonicity similarly; unless otherwise stated, the term "canonical" will mean $A$-canonical.
\end{df}

Choose an open covering $(U_i)$ of some neighborhood of $C$; let $V_i = U_i\cap C$ and $\mc{V}=(V_i)$ the associated open covering of $C$.
Denote by $\mb{C}$ the trivial rank one local system on $C$ and by $\mc{O}_C$ the trivial line bundle on $C$.

First, let us give the following definitions:

\begin{df}
Let $(a_{ij})$ be a cocycle in $Z^1(\mc{V},\mb{C})$ and let $\gamma$ be a loop on $C$.
We define the period of $(a_{ij})$ along $\gamma$ to be the sum
\[
\int_\gamma (a_{ij}) = \sum_{p=1}^n a_{i_p i_{p+1}}
\]
where the open sets $(V_{i_p})_{p=1}^n$ form a simple covering of $\gamma$ and $V_{i_p}\cap V_{i_{p+1}}\cap \gamma\neq \emptyset$.
\end{df}

This application only depends on the class $[\gamma]$ of $\gamma$ in the fundamental group of $C$; taking periods along the $\alpha_i$ and $\beta_i$ gives applications
\[
P_A,P_B : Z^1(\mc{V},\mb{C}) \rightarrow \mb{C}^g.
\]
Putting these together gives an application $P: Z^1(\mc{V},\mb{C}) \rightarrow \mb{C}^{2g}$ which induces an injection $P : H^1(\mc{V},\mb{C}) \rightarrow \mb{C}^{2g}$.

On the other hand, the exact sequence
\[
0 \rightarrow \mb{C} \rightarrow \mc{O}_C \rightarrow \Omega^1 \rightarrow 0
\]
gives the exact sequence in cohomology 
\begin{equation}
\label{eq_exact_sequence}
0 \rightarrow H^0(C,\Omega^1) \overset{\delta}{\rightarrow} H^1(C,\mb{C}) \rightarrow H^1(C,\mc{O}_C) \rightarrow 0.
\end{equation}
The fact that the arrow $H^1(C,\mb{C}) \rightarrow H^1(C,\mc{O}_C)$ is surjective is an easy consequence of (\cite{ueda}, proposition 1).
We have $\mr{dim}(H^0(C,\Omega^1)) = \mr{dim}(H^1(C,\mc{O}_C)) = g$ and $\mr{dim}(H^1(C,\mb{C}))=2g$ so that $P : H^1(C,\mb{C}) \rightarrow \mb{C}^{2g}$, being injective, is bijective.
It is well-known that a $1$-form whose $A$-periods vanish is zero, so that the application $P_A\circ \delta : H^0(C,\Omega^1) \rightarrow \mb{C}^g$ is a bijection.


Constructing a foliation on $S$ is equivalent to constructing functions $y_i$ on $U_i$ which are reduced equations of $C\cap U_i$ such that
\[
y_j = \varphi_{ji}(y_i),
\]
where the $\varphi_{ji}$ are diffeomorphisms of $(\mb{C},0)$.
As before, if $\gamma$ is a loop, we can define the product 
\[
H_{\gamma}((\varphi_{ji})) = \varphi_{i_1i_n}\circ \ldots\circ \varphi_{i_3i_2}\circ \varphi_{i_2i_1}
\]
which will be the holonomy of the foliation given by the $y_i$ along the loop $\gamma$.
To construct such functions, we are going to proceed by steps, but first, we need another definition.

\begin{df}
A set of functions $(y_i)$ on the open sets $U_i$ is called \emph{$A$-normalized at order $\mu$} if the $y_i$ are regular functions on $U_i$ vanishing at order $1$ on $C$ and 
\begin{equation}
\label{eq_mu_foliated}
y_j = \varphi_{ji}^{(\mu)}(y_i) + a_{ji}^{(\mu+1)}y_i^{\mu+1},
\end{equation}
on $U_i\cap U_j$,
where $a_{ji}^{(\mu+1)}$ is a function on $U_i\cap U_j$, the $\varphi_{ji}^{(\mu)}$ are polynomials of degree $\mu$ which are also diffeomorphisms tangent to identity and the holonomies $H_{\alpha_k}((\varphi_{ji}))$ are the identity modulo $y_i^{\mu+1}$ for all $k=1,\ldots,g$.
\end{df}


The idea of the proof is first to construct some functions $(y_i)$ which are $A$-normalized at order $1$, and then to show that every $A$-normalized at order $\mu$ set of functions $(y_i)$ can be transformed into an $A$-normalized at order $(\mu+1)$ set of functions by changes of coordinates $y_i \mapsto y_i-b_iy_i^{\mu+1}$ for some functions $b_i$ on $U_i$.
At the limit, we will thus obtain a formal foliation on $S$ with trivial holonomy along $A$-loops.

\begin{lem}
There exists an $A$-normalized at order $1$ set of functions and the foliations associated to two such sets of functions coincide at order $1$.
\end{lem}

\begin{proof}
Take any reduced equations $(y_i)$ of $C$ and compute $y_j$ in the coordinate $y_i$:
\[
y_j = a_{ji}^{(1)}y_i.
\]
The cocycle $(a_{ji}^{(1)}\vert_C)$ defines the normal bundle $N_C=\mc{O}_C$ of $C$ so is cohomologous to the trivial cocycle: there exist functions $b_i$ on $U_i$ such that
\[
a_{ji}^{(1)}\vert_C = \frac{b_j\vert_C }{b_i\vert_C }.
\]
Put $z_i = y_i/b_i$ to obtain 
\[
z_j = z_i + a_{ji}^{(2)} z_i^2.
\]
for some functions $a_{ji}^{(2)}$.

For unicity, consider two sets of functions $(y_i)$ and $(z_i)$ $A$-normalized at order $1$.
Then $(y_i)$ and $(z_i)$ define two sections $y^1$ and $z^1$ on the normal bundle $N_C$.
Necessarily, $y^1$ and $z^1$ are colinear, hence the result.
\end{proof}

%
%

\begin{lem}
\label{lem_step}
Let $(y_i)$ be a set of functions $A$-normalized at order $\mu$.
Then there exist functions $(b_i)$ on $U_i$ such that the coordinates $z_i = y_i - b_iy_i^{\mu+1}$ are $A$-normalized at order $\mu+1$.
Moreover, two sets of functions $A$-normalized at order $(\mu+1)$ which coincide at order $\mu$ define the same foliation at order $\mu+1$.
\end{lem}

\begin{proof}
Since $(y_i)$ is $A$-normalized at order $\mu$, it satisfies 
\[
y_j = \varphi_{ji}^{(\mu)}(y_i) + a_{ji}^{(\mu+1)} y_i^{\mu+1}.
\]
In the following, denote by $\mr{Diff}^{1}_{\mu}(\mb{C},0) = \mr{Diff}^1(\mb{C},0) / \mr{Diff}^{\mu+1}(\mb{C},0)$ the group of $\mu$-jets of diffeomorphisms tangent to the identity.
The tuple $(\varphi_{ji}^{(\mu)})_{ji}$ is a cocycle in $H^1(C,\mr{Diff}^{1}_{\mu}(\mb{C},0))$; it is entirely determined by its holonomy representation $H((\varphi_{ji}^{(\mu)})) : \pi_1(C) \rightarrow \mr{Diff}^{1}_{\mu}(\mb{C},0)$.
We would like to extend this cocycle to some cocycle in $H^1(C,\mr{Diff}^{1}_{\mu+1}(\mb{C},0))$.
Since $H_{\alpha_k}((\varphi_{ji}^{(\mu)}))$ is trivial for $k=1,\ldots,g$, extend $H_{\alpha_k}((\varphi_{ji}^{(\mu)}))$ to $\rho_{\alpha_k}=id$.
Next, extend the diffeomorphisms $H_{\beta_{k}}((\varphi_{ji}^{(\mu)}))$ to diffeomorphisms $\rho_{\beta_k}\in \mr{Diff}^{1}_{\mu+1}$ in any way.
Then $\prod_{k=1}^g[\rho_{\alpha_k},\rho_{\beta_k}]=id$ so the $(\rho_\gamma)_{\gamma}$ define a representation of $\pi_1(C)$ into $\mr{Diff}^{1}_{\mu+1}$ which corresponds to a cocycle $(\psi_{ji})$ such that $H_{\alpha_k}((\psi_{ji}))=\rho_{\alpha_k}$ and $H_{\beta_k}((\psi_{ji}))=\rho_{\beta_k}$.
We can then write
\[
y_j = \psi_{ji}(y_i) + {a'_{ji}}^{(\mu+1)} y_{i}^{\mu+1}
\]
for some ${a'_{ji}}^{(\mu+1)}$.
Next,
\begin{align*} 
y_k &= \psi_{kj}(y_j) + {a'_{kj}}^{(\mu+1)} y_{j}^{\mu+1}\\
    &= \psi_{kj} \left( \psi_{ji}(y_i) + {a'_{ji}}^{(\mu+1)} y_{i}^{\mu+1} \right) + {a'_{kj}}^{(\mu+1)} \left(\psi_{ji}(y_i) + {a'_{ji}}^{(\mu+1)} y_{i}^{\mu+1} \right)^{\mu+1}\\
    &= \psi_{kj} ( \psi_{ji}(y_i)) + ({a'_{ji}}^{(\mu+1)}+{a'_{jk}}^{(\mu+1)} ) y_{i}^{\mu+1} + \ldots
\end{align*}

Since $\psi_{ki} = \psi_{kj}\psi_{ji}$, we obtain ${a'_{ki}}^{(\mu+1)}\vert_C = {a'_{kj}}^{(\mu+1)}\vert_C + {a'_{ji}}^{(\mu+1)}\vert_C$ and thus $({a'_{ji}}^{(\mu+1)}\vert_C)$ is a cocycle in $H^1(C,\mc{O}_C)$.
By the exact sequence \eqref{eq_exact_sequence}, it is cohomologous to a constant cocycle $(c_{ji})\in H^1(C,\mb{C})$: there exists functions $b_i$ on $U_i$ such that ${a'_{ji}}^{(\mu+1)}\vert_C - c_{ji} = b_j\vert_C - b_i\vert_C$.
Still using the exact sequence \eqref{eq_exact_sequence}, we see that two cocycles $(c_{ji}),(c'_{ji})$ cohomologous to $({a'_{ji}}^{(\mu+1)}\vert_C)$ differ only by the periods of a $1$-form.
As noted before, $P_A\circ \delta: H^0(C,\Omega^1) \rightarrow H^1(C,\mb{C})$ is bijective so we can choose $(c_{ji})$ with trivial $A$-periods, and such a $(c_{ji})$ is unique.
Put $\varphi_{ji}^{(\mu+1)}(y) = \psi_{ji}(y) + c_{ji} y^{\mu+1}$ and $z_i = y_i - b_i y_i^{\mu+1}$ to obtain 
\begin{align*}
z_j &= \psi_{ji}(z_i) + ({a'_{ji}}^{(\mu+1)} -b_j+b_i) z_i^{\mu+1} + o(z_i^{\mu+1})\\
    &= \varphi_{ji}^{(\mu+1)}(z_i) + o(z_i^{\mu+1}).
\end{align*}
Since the choice of $(c_{ji})\in H^1(C,\mb{C})$ is unique, if two sets of functions $(z_i), (z'_i)$ are both $A$-normalized at order $\mu+1$ and coincide at order $\mu$, then they differ at order $\mu+1$ by a coboundary $(d_i)\in H^0(C,\mb{C})$: $z'_i = z_i + d_iz_i^{\mu+1} + \ldots$
Hence, they define the same foliation at order $\mu+1$.
\end{proof}

Putting all this together, we obtain theorem \ref{thm_constr_fol}.

\section{Classification of bifoliated neighborhoods}
\label{sec_bif_classification}

A \emph{bifoliated neighborhood} of $C$ is a tuple $(S,\F,\G)$ where $S$ is a neighborhood of $C$ and $\F$, $\G$ are distinct foliations on $S$ having $C$ as a common leaf.
Two bifoliated neighborhoods $(S,\F,\G)$ and $(S',\F',\G')$ are said to be equivalent if there are two neighborhoods $U\subset S$, $U'\subset S'$ of $C$ and a diffeomorphism $\phi : U \rightarrow U'$ fixing $C$ such that
\[
\phi_* \F = \F'\quad \text{and}\quad \phi_*\G = \G'.
\]
In this section, we want to study the classification of bifoliated neighborhoods under this equivalence relation.
We will consider here analytic equivalence, but the formal classification can be obtained by replacing the word "analytic" by "formal" everywhere.


A neighborhood will have a lot a formal foliations, and the canonical ones may diverge even though others might converge (cf. \cite{ltt}).
We will thus consider here a general bifoliated neighborhood $(S,\F,\G)$, with the additional assumptions that $\F$ and $\G$ coincide at order $1$ and their holomy representations are tangent to the identity.
The study can be done without these assumptions (cf. \cite{thom_these}), but they are true for the pair of canonical foliations and simplify the results (for example, in general an affine structure is involved which under our assumptions is only a translation structure, ie. a $1$-form).

\subsection{First invariants}

If $(S,\F,\G)$ is a bifoliated neighborhood, each foliation comes with the holonomy representation of the leaf $C$:
\[
\rho_\F,\rho_\G : \pi_1(C) \rightarrow \Df{},
\]
Fix a base point $p_0\in C$, a transversal $T_0$ passing through $C$ at $p_0$ and a coordinate $t$ on $T_0$ (ie. a function $t\in \Cg{} \mapsto q(t)\in T_0$.
Let $\gamma$ be a loop on $C$ based at $p_0$; choose the minimal first integral $F$ of $\F$ around $T_0$ such that $F(q(t)) = t$.
The analytic continuation $F^\gamma$ of $F$ along $\gamma$ is again a first integral of $\F$, hence is of the form 
\[
F^\gamma = \varphi_{\gamma}^{-1}\circ F.
\]
We define $\rho_\F(\gamma) = \varphi_{\gamma}$.

A second invariant is the order of tangency between $\F$ and $\G$ along $C$: take two $1$-forms $\alpha$ and $\beta$ on $S$ defining locally the foliations $\F$ and $\G$.
The $2$-form $\alpha \wedge \beta$ vanishes on $C$ so the order of vanishing of $\alpha \wedge \beta$ along $C$ gives a global invariant $k+1$ which does not depend on the choice of $\alpha$ and $\beta$.
The order of tangency between $\F$ and $\G$ is defined to be this integer $k$.
Our assumption that the foliations coincide at order $1$ exactly means that $k\geq 1$.

The next invariant is a $1$-form on $C$ associated to this pair of foliations.
Choose as before a point $p_0\in C$, a transversal $T_0$ at $p_0$ and a coordinate $t \mapsto q(t)$ on $T_0$.
Take local minimal first integrals $F$ and $G$ of $\F$ and $\G$ such that $F(q(t)) = G(q(t)) = t$.
By definition of $k$, $G=F+aF^{k+1}+\ldots$ in a neighborhood of $p$ for a local function $a$ on $C$.
Take the analytic continuations $F^\gamma, G^\gamma$ and $a^\gamma$ of $F,G$ and $a$ along a loop $\gamma$.
Then we can use the fact that the holonomy representations of $\F$ and $\G$ are tangent to the identity to get
\begin{align*}
G^\gamma &= F^\gamma + (a^\gamma) (F^\gamma)^{k+1} + \ldots\\
\rho_\G(\gamma)^{-1}\circ G &= \rho_\F(\gamma)^{-1}\circ F + (a^\gamma)(\rho_\F(\gamma)^{-1}\circ F)^{k+1} + \ldots\\
G &= \rho_\G(\gamma) \left(\rho_\F(\gamma)^{-1}\circ F + a^\gamma F^{k+1} + \ldots \right)\\
&= \rho_\G(\gamma)\circ \rho_\F(\gamma)^{-1}\circ F + a^\gamma F^{k+1}+\ldots\\
&= F + a F^{k+1} + \ldots
\end{align*}
Since $\rho_\G(\gamma)\circ \rho_\F(\gamma)^{-1}$ has constant coefficients, there exists constants $c^\gamma$ such that $a^\gamma = a+c^\gamma$.
Then the $1$-form $\omega = da$ is a well-defined $1$-form on $C$.

Note that we saw in the process that
\begin{equation}
\label{eq_compatibility_1}
\rho_\G(\gamma)\circ \rho_\F(\gamma)^{-1}(y) = y + \left(\int_{\gamma}\omega \right) y^{k+1}+\ldots,
\end{equation}
thus the form $\omega$ is entirely determined by $\rho_\F$ and $\rho_\G$.

Note also that the holonomy representation and the form $\omega$ depend on the choice of the transversal $T_0$ and of a coordinate $t$ on it.
A change of coordinate $\tilde{t} = \varphi(t)$ induces conjugacies on $\rho_\F$ and $\rho_\G$ and changes $\omega$ into some multiple of it: $\tilde{\rho}_\F(\gamma) = \varphi\circ \rho_\F(\gamma)\circ \varphi^{-1}$, $\tilde{\rho}_\G(\gamma) = \varphi\circ \rho_\G(\gamma)\circ \varphi^{-1}$ and $\tilde{\omega} = \varphi'(0)^{-k} \omega$.

\subsection{Tangency set}

If $F$ and $G$ are local minimal first integrals of $\F$ and $\G$, then the tangency set between $\F$ and $\G$ is defined to be
\[
\mr{Tang}(\F,\G) = \{ dF \wedge dG = 0\}.
\]
This definition does not depend on the choice of $F$ and $G$ and gives a well-defined analytic subset of $S$.

Note that if we write $G = F+aF^{k+1}+\ldots$, then we obtain $dF \wedge dG = F^{k+1} dF \wedge (da + \ldots)$.
Since $\omega=da$,
\[
\mr{Tang}(\F,\G)\cap C = \{\omega = 0\}.
\]
In particular, the set $\mr{Tang}(\F,\G)$ intersects $C$ at $2g-2$ points counting multiplicities.
In the sequel, we will suppose that we are in the generic case: $\omega$ has $2g-2$ distinct zeroes.
This also means that $\mr{Tang}(\F,\G)$ is the union of $2g-2$ curves which are transverse to $C$.

Denote $p_1,\ldots,p_{2g-2}$ the zeroes of $\omega$ and $T_i$ the tangency curve passing through $p_i$.
If we fix some simple paths $\gamma_{ij}$ between $p_i$ and $p_j$, we can look at the holonomy transports
\[
\varphi_{ij}^{\F},\varphi_{ij}^{\G}: T_i \rightarrow T_j
\]
following the leaves of $\F$ and $\G$ along $\gamma_{ij}$.
To simplify, suppose that the $\gamma_{1j}$ only intersect each other at $p_1$ and that $\gamma_{ij} = \gamma_{1i}^{-1}\cdot \gamma_{1j}$.
\begin{figure}[H]
\includegraphics[scale=0.5]{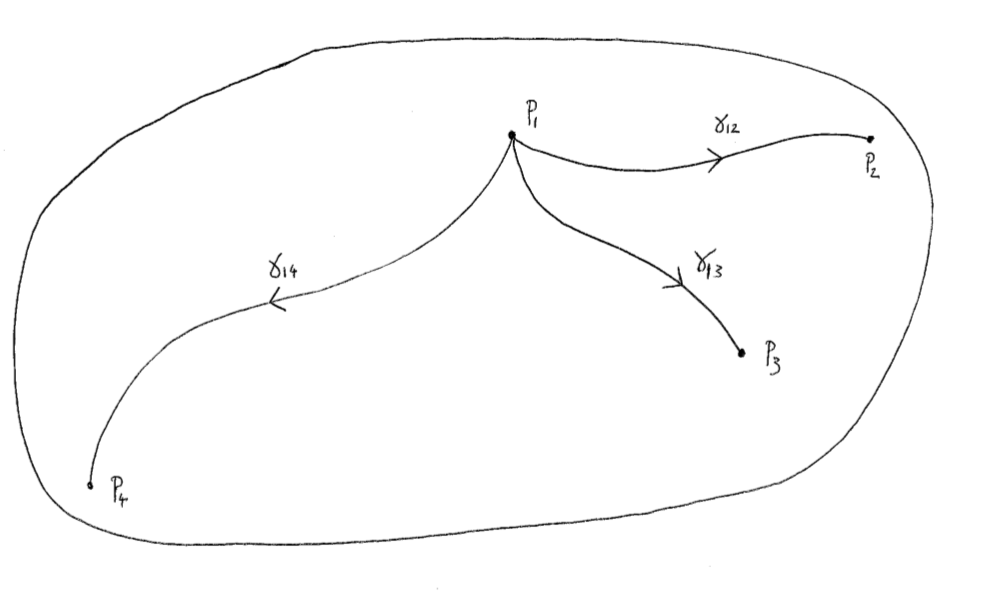}
\caption{The paths $\gamma_{1j}$}
\end{figure}

To define this, fix some coordinates $t_i,t_j$ on $T_i$ and $T_j$; there exists a simply connected neighborhood $U$ of the path $\gamma_{ij}$.
Let $L$ be a leaf of $\F$ on $U$.
It intersects $T_i$ at exactly one point (let $t_i$ be its coordinate).
In the same way, let $t_j$ be the coordinate of $L\cap T_j$.
We set $\varphi_{ij}^{\F}(t_i)=t_j$: this gives a germ of diffeomorphism $\varphi_{ij}^{\F}\in \Df{}$ which depends on the choices of coordinates on $T_i$ and $T_j$.
Their composition
\[
\varphi_{ij}^{\leftrightarrow} = (\varphi_{ij}^{\G})^{-1}\circ \varphi_{ij}^{\F}
\]
is a diffeomorphism of $T_i$ so only depends on the choice of a coordinate on $T_i$; a change of coordinate $t'_i=\varphi(t_i)$ acts by conjugacy $\varphi_{ij}'^{\leftrightarrow} = \varphi\circ \varphi_{ij}^{\leftrightarrow}\circ \varphi^{-1}$.

We can show as in the previous subsection that the holonomy transports $\varphi_{ij}^{\leftrightarrow}$ are related to the $1$-form $\omega$ by the relation:
\begin{equation}
\label{eq_compatibility_2}
\varphi_{ij}^{\leftrightarrow}(t_i) = t_i - \left(\int_{\gamma_{ij}}\omega \right) t_i^{k+1} + \ldots
\end{equation}


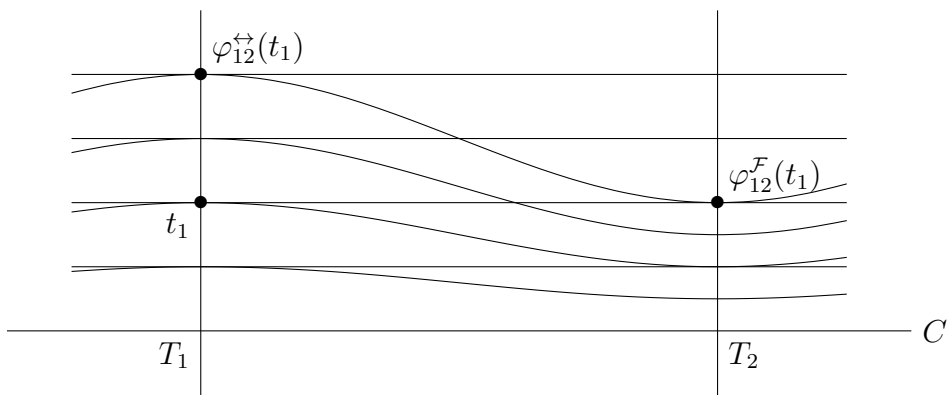
\begin{figure}[H]
\begin{center}
\begin{tikzpicture}[scale = 1.7]
\draw (-1.5,0) -- (5.5,0);
\draw (-1,0.5) -- (5,0.5);
\draw (-1,1) -- (5,1);
\draw (-1,1.5) -- (5,1.5);
\draw (-1,2) -- (5,2);
\draw (0,-0.5) -- (0,2.5);
\draw (4,-0.5) -- (4,2.5);
\draw plot [domain = -1:5,smooth] (\x,{1/2*(cos(180/4*(\x))-1)+2});
\draw plot [domain = -1:5,smooth] (\x,{1.5/4*(cos(180/4*(\x))-1)+1.5});
\draw plot [domain = -1:5,smooth] (\x,{1/4*(cos(180/4*(\x))-1)+1});
\draw plot [domain = -1:5,smooth] (\x,{1/8*(cos(180/4*(\x))-1)+0.5});
\draw (0,1) node {$\bullet$};
\draw (0,1) node[below left] {$t_1$};
\draw (4,1) node {$\bullet$};
\draw (4,1) node[above right] {$\varphi_{12}^\F(t_1)$};
\draw (0,2) node {$\bullet$};
\draw (0,2) node[above right] {$\varphi_{12}^{\leftrightarrow}(t_1)$};
\draw (0,0) node[below left] {$T_1$};
\draw (4,0) node[below right] {$T_2$};
\draw (5.5,0) node[right] {$C$};
\end{tikzpicture}
\end{center}
\caption{Holonomy transports}
\end{figure}

\subsection{Classification of bifoliated neighborhoods}

We say that a bifoliated neighborhood $(S,\F,\G)$ is generic if there are $2g-2$ distinct tangency curves $T_1,\ldots,T_{2g-2}$ between $\F$ and $\G$ and if they intersect $C$ transversely at some distinct points $p_1,\ldots,p_{2g-2}$.

On each neighborhood, we can fix one of these points, for example $p_1$, fix a coordinate $t$ on $T_1$, fix paths $\gamma_{1j}$ between $p_1$ and $p_j$ and compute every invariant on the transversal $T_1$ with coordinate $t$.
We thus have the holonomy representations $\rho_\F, \rho_\G$ and the holonomy transports $\varphi_{1j}^{\leftrightarrow}$ between $T_1$ and another tangency curve $T_j$.

The holonomy representations $\rho_\F$ and $\rho_\G$ are entirely determined by the images of the basis $\alpha_1,\ldots,\alpha_g,\beta_1,\ldots,\beta_g$: these are any diffeomorphisms such that

\begin{equation}
\label{eq_representation}
\begin{aligned}
\text{ }
[\rho_\F(\alpha_1),\rho_\F(\beta_1)]\ldots[\rho_\F(\alpha_g),\rho_\F(\beta_g)] &= id\\
[\rho_\G(\alpha_1),\rho_\G(\beta_1)]\ldots[\rho_\G(\alpha_g),\rho_\G(\beta_g)] &= id.
\end{aligned}
\end{equation}

%

Every invariant diffeomorphism found $\varphi = \rho_\F(\alpha_i), \rho_\F(\beta_i), \rho_\G(\alpha_i), \rho_\G(\beta_i), \varphi_{1j}^\leftrightarrow$ depend on the choice of the coordinate $t$.
A change of coordinate $t' = \psi(t)$ induces a conjugacy on $\varphi$: $\varphi' = \psi\circ \varphi\circ \psi^{-1}$.
So we define the invariant of a neighborhood $(S,\F,\G)$ to be
\begin{align*}
Inv(S,\F,\G) &= \left[ ((\rho_\F(\alpha_i))_{i=1}^{g},(\rho_\F(\beta_i))_{i=1}^g,(\rho_\G(\alpha_i))_{i=1}^{g},(\rho_\G(\beta_i))_{i=1}^g,(\varphi_{1j}^{\leftrightarrow})_{j=2}^{2g-2}) \right]\\
	     &\in \Df{}^{2g}\times\Df{}^{2g}\times\Df{}^{2g-3}/\sim
\end{align*}
where $\sim$ is the action of $\Df{}$ by conjugacy on each factor.

%

\begin{thm}
\label{thm_classif_bif}
Let $C$ be a curve of genus $g\geq 2$. 
Let $(S,\F,\G)$ and $(S',\F',\G')$ be two bifoliated neighborhoods of $C$ with 
tangency order $k$ and $1$-form $\omega$.
Suppose $k\geq 1$ and that $\omega$ has simple zeroes $p_1,\ldots,p_{2g-2}$.
Denote $T_1, T'_1$ the tangency curves passing through $p_1$ and compute the invariants $Inv(S,\F,\G)$ and $Inv(S',\F',\G')$ on the tangency curves $T_1, T'_1$.

Then $(S,\F,\G)$ and $(S',\F',\G')$ are diffeomorphic if and only if
\[
Inv(S',\F',\G')=Inv(S,\F,\G).
\]
\end{thm}

Before starting the proof, let us write some lemmas.

\begin{lem}
\label{lem_fcts_transv}
Let $p$ be a point in $C$, let $F$, $G$ be two reduced equations of $C$ around $p$ and $(x,y)$ some local coordinates with $C=\{y=0\}$.
Suppose that $F$ and $G$ are tangent at order $k$ and that the zero divisor of $dF \wedge dG$ is $(k+1)C$ (ie. there are no other tangencies).
There exists a unique diffeomorphism $\phi$ fixing $C$ pointwise such that
\[
(F,G)\circ \varphi = (y,y+a(x)y^{k+1}).
\]
The function $a$ is unique and satisfies $da\vert_C= \omega$.
\end{lem}

The proof of this lemma can be found in \cite{ltt}.

%

\begin{lem}
\label{lem_fcts_tang}
Let $p$ be a point in $C$, let $F$, $G$ be two reduced equations of $C$ around $p$ and $(x,y)$ some local coordinates with $C=\{y=0\}$.
Suppose that there is a transversal $T$ to $C$ such that the zero divisor of $dF \wedge dG$ is $(k+1)C+T$.
Then there exists a unique diffeomorphism $\phi$ fixing $C$ pointwise such that
\[
(F,G)\circ \varphi = (y,b(y) + a(x)y^{k+1}).
\]
The function $b$ is unique and $a$ is the primitive of $\omega$ which is zero at $p$.
\end{lem}

The function $b$ is of course entirely determined by the equation $G\vert_T = b(F\vert_T)$.

\begin{proof}
Put $\tilde{y}=F$, $b$ the function determined by $G\vert_T = b(F\vert_T)$, $H = G-b(\tilde{y})$ and suppose $x$ is a reduced equation of $T$.
Then $dF \wedge dG = (\partial_{x}H)dF \wedge dx$ so by the hypotheses on the tangency divisor, $\partial_{x}H = 2x\tilde{y}^{k+1}u$ for some invertible function $u$.
Then $H = x^2\tilde{y}^{k+1}v$ with $v$ invertible so for $\phi(x,\tilde{y}) = x\sqrt{v}$, we have $G = b(\tilde{y}) + \phi(x,\tilde{y})^2\tilde{y}^{k+1}$.

If $\psi = \phi\vert_C$, then the coordinate $\tilde{x}=\psi^{-1}\circ \phi(x,\tilde{y})$ is equal to $x$ on $C$ and $(F,G) = (\tilde{y},b(\tilde{y}) + \psi(\tilde{x})^2\tilde{y}^{k+1})$.
Thus the diffeomorphism $\varphi(x,y) = (\tilde{x},\tilde{y})$ is as sought.
\end{proof}

\begin{lem}
\label{lem_hol_transport}
Let $(S,\F,\G)$ be a bifoliated neighborhood whose $1$-form $\omega$ has simple zeroes, let $T_1,T_j$ be two tangency curves and $\gamma_{1j}$ a simple path between $p_1=T_1\cap C$ and $p_j=T_j\cap C$.
Suppose $F$ and $G$ are some submersive first integrals of $\F$ and $\G$ around $p_1$ such that $F\vert_{T_1} = G\vert_{T_1}$.
By lemma \ref{lem_fcts_tang}, the analytic continuations of $F$ and $G$ along $\gamma_{1j}$ can be written $F=y$ and $G=b(y) + a(x)y^{k+1}$ for some coordinates $(x,y)$ around $p_j$.

Then $b(y) = \varphi_{1j}^{\leftrightarrow}(y)$ if $\varphi_{1j}^{\leftrightarrow}$ is computed in the coordinate $t=y$ on $T_1$.
\end{lem}

\begin{proof}
Indeed, $b$ is characterised by $G\vert_{T_j} = b\circ F\vert_{T_j}$, and $\varphi_{1j}^{\leftrightarrow}$ by the fact that the leaf of $\F$ passing through $T_1$ at the point of coordinate $F=y_0$ intersects (tangentially) on $T_j$ the leaf of $\G$ passing through $T_1$ at the point of coordinate $F=\varphi_{1j}^{\leftrightarrow}(y_0)$.
This means that the first integral $\varphi_{1j}^{\leftrightarrow}\circ F$ of $\F$ coincides with $G$ on $T_j$, ie $\varphi_{1j}^{\leftrightarrow}\circ F\vert_{T_j} = G\vert_{T_j}$, hence the result.
\end{proof}

\begin{proof}[Proof of theorem \ref{thm_classif_bif}]
Take two bifoliated neighborhoods $(S,\F,\G)$ and $(S',\F',\G')$ with the same tangency index $k$, $1$-form $\omega$ and the same invariants computed in some coordinates $t,t'$ on $T_1$ and $T'_1$.

Begin by fixing simply connected neighborhoods $Y$, $Y'$ of $\cup_{j=2}^{2g}\gamma_{1j}$ in $S$ and $S'$.
We begin by showing that $(Y,\F,\G)$ and $(Y',\F',\G')$ are diffeomorphic, and we will then show that this diffeomorphism can be extended to $S$ and $S'$.
\begin{figure}[H]
\includegraphics[scale=0.75]{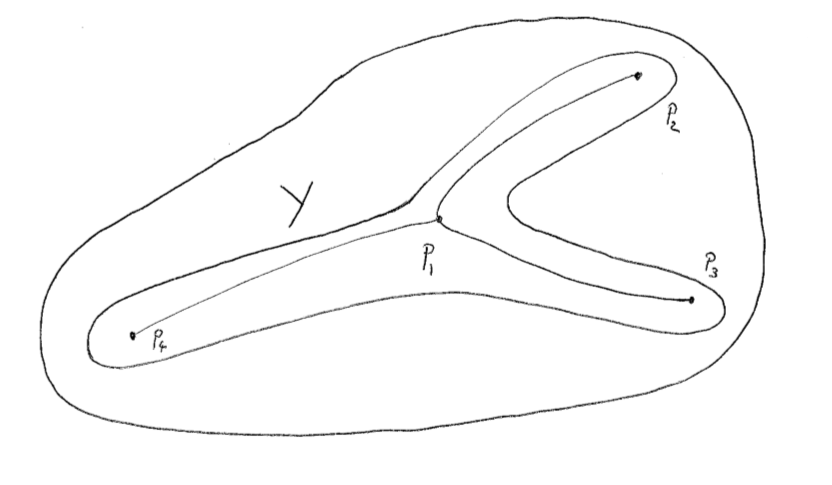}
\caption{The neighborhood $Y$}
\end{figure}

Since $Y$ and $Y'$ are simply connected, the foliations on these sets have first integrals $F,G,F',G'$ and we can suppose that $F(t)=G(t)$ and $F'(t')=G'(t')$ on $T_1$ and $T'_1$.
Since $\omega=\omega'$, lemma \ref{lem_fcts_tang} tells us that there is a (unique) diffeomorphism $\psi$ between a neighborhood of $p_1$ in $S$ and a neighborhood of $p_1$ in $S'$ such that $F'\circ \psi = F$ and $G'\circ \psi=G$.
We can take the analytic continuation of $F,G,F'$ and $G'$ along one of the paths $\gamma_{1j}$.
For any point $p$ in this path, lemma \ref{lem_fcts_transv} tells us that the pairs $(F,G)$ and $(F',G')$ are equivalent (by a unique diffeomorphism) if and only if the number $a(p)$ is the same for both couples.
But $a(p) = \int_{p_1}^p \omega$ where the integral is taken along the path $\gamma_{1j}$ so it is the case.
By unicity, the diffeomorphism $\psi$ can be extended along the path $\gamma_{1j}$ arbitrarily near the point $p_j$.

At the point $p_j$, lemmas \ref{lem_fcts_tang} and \ref{lem_hol_transport} show that the pairs $(F,G)$ and $(F',G')$ are also conjugated by a unique diffeomorphism in a neighborhood of $p_j$.
Hence, we can extend $\psi$ to a diffeomorphism $\psi:Y \rightarrow Y'$ conjugating the pairs $(F,G)$ and $(F',G')$.

By the lemma \ref{lem_fcts_transv}, we can also extend $\psi$ along any simple path.
Then we only need to show that $\psi$ can be extended along a non-trivial loop.
Let $\gamma$ be a non-trivial loop on $C$ based at $p_1$, $\varphi_\F = \rho_\F(\gamma)$ and $\varphi_\G = \rho_\G(\gamma)$.
The extensions of $F$ and $G$ along $\gamma$ are $\varphi_{\F}^{-1}\circ F$ and $\varphi_\G^{-1}\circ G$; we know that $F'\circ \psi=F$ and $G'\circ \psi=G$, so $\varphi_\F^{-1}\circ F'\circ \psi = \varphi_\F^{-1}\circ F$ and $\varphi_\G^{-1}\circ G'\circ \psi = \varphi_\G^{-1}\circ G$.
Hence $\psi$ is the diffeomorphism conjugating $(\varphi_\F^{-1}\circ F,\varphi_\G^{-1}\circ G)$ with $(\varphi_\F^{-1}\circ F',\varphi_\G^{-1}\circ G')$ and by unicity this means that $\psi$ can be extended along any loop.
Thus $\psi$ can be extended to a diffeomorphism between $(S,\F,\G)$ and $(S',\F',\G')$.
\end{proof}

\subsection{Construction of bifoliated neighborhoods}

We saw three restrictions for a set of diffeomorphisms to be an invariant of some bifoliated neighborhood: these are the compatibility relations \eqref{eq_compatibility_1}, \eqref{eq_compatibility_2} and \eqref{eq_representation}.
These are the only restrictions; to obtain a simpler result, we will consider the $1$-form $\omega$ as an invariant here.

\begin{thm}
\label{thm_constr_bif}
Let $((\varphi_i^1)_{i=1}^{2g},(\varphi_i^2)_{i=1}^{2g},(\varphi_j^3)_{j=2}^{2g-3})$ be some diffeomorphisms; let $k$ be an integer and $\omega$ a $1$-form.
They define a bifoliated neighborhood $(S,\F,\G)$ with $\F$ and $\G$ tangent at order $k$ and with $1$-form $\omega$ if and only if every $\varphi_r^s$ is tangent to identity at order (at least) $k$ and if they satisfy the relations \eqref{eq_compatibility_1}, \eqref{eq_compatibility_2} and \eqref{eq_representation}.
\end{thm}

\begin{proof}
Denote by $\rho_1$ and $\rho_2$ the representations given by the diffeomorphisms $(\varphi_i^1)$ and $(\varphi_i^2)$.
Consider $\tilde{C}=\mb{D}_x$ the universal cover of $C$, $X$ a small neighborhood of a fundamental domain, $U_i$ a small neighborhood of $p_i$ in $X$ and $\check{C} = X\setminus(U_2\cup\ldots\cup U_{2g-2})$.
\begin{figure}[H]
\includegraphics[scale=0.6]{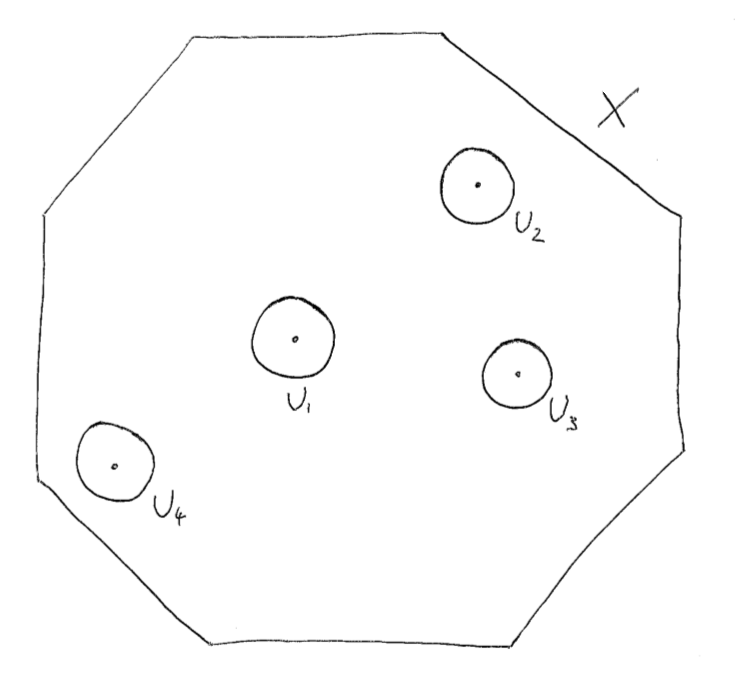}
\caption{The neighborhood $X$}
\end{figure}

Consider next the trivial bundle $\check{S} = \check{C}\times \mb{C}_y$ along with two functions $F=y$ and $G = y + a(x)y^{k+1}$ (with $a(x) = \int_{p_1}^x \omega$).
We now want to glue the borders of $\check{S}$ together: for this, we need to show that there exists for each loop $\gamma$ a diffeomorphism $\psi_{\gamma}$ defined when it makes sense such that $\psi_{\gamma}\vert_{\check{C}} = \gamma$ and
\[
(\rho_1(\gamma)\circ F, \rho_2(\gamma)\circ G) = (F\circ \psi_{\gamma},G\circ \psi_{\gamma}).
\]
Thanks to the compatibility condition \eqref{eq_compatibility_1} and lemma \ref{lem_fcts_transv}, the couples $(\rho_1(\gamma)\circ F,\rho_2(\gamma)\circ G)$ and $(F\circ \gamma,G\circ \gamma)$ are diffeomorphic so we can indeed find such a $\psi_{\gamma}$.
We can then glue the borders of $\check{S}$ together to obtain a surface which is a neighborhood of $C$ with holes $H_i$ around $p_i$ ($i=2,\ldots,2g-2$) and two foliations $\F$ and $\G$ transverse outside the holes.
The holonomies of these foliations are $\rho_\F = \rho_1$ and $\rho_\G = \rho_2$ by construction.

To fill these holes, take $C_i$ a neigrborhood of $p_i$ in $X$ slightly larger than $U_i$ and consider the patch $P_i = C_i\times \mb{C}_y$.
Consider on $P_i$ the couple 
\[
(\tilde{F},\tilde{G}) = (y,\varphi_i^3(y)+(x-p_i)^2y^{k+1}).
\]
By lemma \ref{lem_fcts_transv} and compatibility condition \eqref{eq_compatibility_2}, for every point $p$ near the boundary of the hole $H_i$, there exists a unique diffeomorphism $\psi$ between a neighborhood of $p$ in $\check{S}$ and a neighborhood of $p$ in $P_i$ sending $(F,G)$ to $(\tilde{F},\tilde{G})$.
By unicity, these diffeomorphisms glue to a diffeomorphism between neighborhoods of the boundaries of $H_i$ and $P_i$ and we can then glue the patch $P_i$ onto $H_i$ using this diffeomorphism.
By the lemma \ref{lem_hol_transport}, we then have $\varphi_{1i}^{\leftrightarrow} = \varphi_i^3$ which concludes the proof.
\end{proof}

\section{Formal classification of neighborhoods}
\label{sec_formal_classification}

We know how to construct two canonical foliations on any neighborhood, and we know the classification of bifoliated neighborhoods, so we only need to put this together.

Denote $\F$ and $\G$ the $A$- and $B$-canonical foliations.
Note that if $\F=\G$, then they define a fibration tangent to $C$, so this case can be treated by Kodaira's deformation theory.
Suppose this is not the case and $\F\neq\G$, then their order of tangency $k$ is the Ueda index of $S$.
Moreover, let $(u_{ij})\in H^1(C,\mc{O}_C)$ be the Ueda class of the neighborhood and $(a_{ij}),(b_{ij})\in H^1(C,\mb{C})$ the cocycles defining the $(k+1)$-th order holonomy of $\F$ and $\G$.
By definition, the images of $(a_{ij})$ and $(b_{ij})$ under the map $H^1(C,\mb{C}) \rightarrow H^1(C,\mc{O}_C)$ are both $(u_{ij})$.
Thus by the exact sequence \eqref{eq_exact_sequence}, the cocycle $(b_{ij}-a_{ij})$ is given by a $1$-form: this $1$-form is exactly $\omega$.
To sum up, we have constructed an application $H^1(C,\mc{O}_C) \rightarrow H^0(C,\Omega^1)$; this is a bijection because we can find $(a_{ij})$ (and thus $(u_{ij})$) from $\omega$ as the cocycle with null $A$-periods and with $B$-periods equal to those of $\omega$.

By extension, we will call this form the Ueda form of the neighborhood.
The Ueda class (and thus the Ueda form) is well-defined only up to a multiplicative constant, but the set of its zeroes is well-defined.
The situation will be quite different depending on the tangency set between $\F$ and $\G$, so suppose that $\omega$ has only simple zeroes (so that the tangency set consists of $2g-2$ simple transversal tangency curves).
Denote by $\ms{V}(C,k,\omega)$ the space of $2$-dimensional formal neighborhoods of $C$ with trivial normal bundle, Ueda index $k<\infty$ and Ueda form (a multiple of) $\omega$ modulo formal equivalence.

\begin{thm}
\label{thm_formal_classification}
Let $C$ be a curve of genus $g\geq 2$, $1\leq k<\infty$ and $\omega$ a $1$-form on $C$ with simple zeroes.
Then there is an injective map 
\[
\Phi : \ms{V}(C,k,\omega) \hookrightarrow \Dfhat{}^g\times\Dfhat{}^g\times\Dfhat{}^{2g-3}/\sim
\]
where the equivalence relation $\sim$ is given by the action of $\Df{}$ on $\Df{}^N$ by conjugacy on each factor.
%

A tuple of diffeomorphisms $((\varphi_{i}^{(j)})_i)_{j=1}^3$ is in the image of $\Phi$ if and only if the $\varphi_{i}^{(j)}$ are tangent to the identity at order $k$ and if they satisfy the compatibility conditions \eqref{eq_compatibility_1} and \eqref{eq_compatibility_2}.
\end{thm}


\begin{proof}
Fix a zero $p_1$ of $\omega$, fix some loops $\alpha_1,\ldots,\alpha_g,\beta_1,\ldots,\beta_g$ forming a symplectic basis of $H_1(C,\mb{C})$, fix some paths $\gamma_{1j}$ between $p_1$ and $p_j$.

Let $[S]\in \ms{V}(C,k,\omega)$ and $S$ be a representative of $[S]$.
Let $\F$ and $\G$ be respectively the $A$-canonical and the $B$-canonical foliations on $S$.
%
%
Let $\varphi_{\tau}^{\F}$ and $\varphi_{\tau}^{\G}$ be the holonomies of $\F$ and $\G$ along the loops $\tau=\alpha_1,\ldots,\beta_g$; let $\varphi_{1j}^{\leftrightarrow} = (\varphi_{1j}^{\G})^{-1}\circ \varphi_{1j}^{\F}$ be computed along the path $\gamma_{1j}$.
We put
\[
\theta(S) = ((\varphi_{\beta_i}^{\F})_{i=1}^{g},(\varphi_{\alpha_i}^{\G})_{i=1}^g,(\varphi_{1j}^{\leftrightarrow})_{j=2}^{2g-2})\quad\text{and}\quad \Phi([S]) = [\theta(S)]
\]
the class of $\theta(S)$ modulo common conjugacy.

Since a diffeomorphism $\psi$ between two neighborhoods $S$ and $S'$ sends the $A$-canonical foliation $\F$ of $S$ to the $A$-canonical foliation $\F'$ of $S'$ (resp. the $B$-canonical foliations $\G,\G'$), $\psi$ then sends the bifoliated neighborhood $(S,\F,\G)$ to $(S',\F',\G')$.
Thus $\theta(S)$ and $\theta(S')$ are conjugated, ie. $\Phi([S])$ is well-defined.
Conversely, if $\Phi([S])=\Phi([S'])$, then $(S,\F,\G)$ is diffeomorphic to $(S',\F',\G')$ (and therefore $S$ is diffeomorphic to $S'$).
%

The realization part of the theorem is a direct consequence of theorem \ref{thm_classif_bif} (the relation \eqref{eq_representation} is trivial here).
\end{proof}

\begin{rmq}
About the realization of a tuple $((\varphi_{i}^{(j)})_i)_{j=1}^3$, remark that the conditions \eqref{eq_compatibility_1} and \eqref{eq_compatibility_2} only depend on the coefficients of $\varphi_i^{(j)}$ of order $k+1$.
In this sense, we can say that the image $\Phi(\ms{V}(C,k,\omega))$ is of finite codimension.
\end{rmq}

\section{Concluding remarks}

\subsection{About convergent foliations in $S$}

In some cases, the canonical foliations do not converge even if the neighborhood is analytic.
Indeed, if $C$ is an elliptic curve, Mishustin gave in \cite{mishustin_no_foliation} an example of a neighborhood $S$ of $C$ with trivial normal bundle and no analytic foliations tangent to $C$.

We can use this example to build examples in higher genus: let $p_1,p_2$ be two points on $C$ and $T_1,T_2$ two transversals at $p_1$ and $p_2$.
Consider the two-fold branched covering $\pi : S' \rightarrow S$ of $S$ branching at $T_1$ and $T_2$.
Denote $C' = \pi^{-1}(C)$, $\alpha,\beta$ the $A$- and $B$-loops on $C$ based at $p_1$, and $\alpha_1,\alpha_2,\beta_1,\beta_2$ the preimages of $\alpha$ and $\beta$.
They are the $A$- and $B$-loops on $C'$ based at $\pi^{-1}(p_1)$.
If $\F,\G$ are the canonical foliations on $S$, denote $\F'$ and $\G'$ the preimages of $\F,\G$ by $\pi$.

Then $S'$ is an analytic neighborhood of the genus $2$ curve $C'$, the canonical foliations of $S'$ are $\F'$ and $\G'$, and they do not converge.

Even with these examples, the question of the existence of an analytic neighborhood of a genus $2$ curve without any convergent foliation is still open.

\subsection{About analytic equivalence of neighborhoods}

Let $S,S'\in\ms{V}(C,k,\omega)$ be two analytic neighborhoods such that the canonical foliations converge.
Let $\theta=(\varphi_i),\theta'=(\varphi'_i)$, $i=1,\ldots,4g-3$ be the diffeomorphisms obtained in the construction, so that $\theta$ is a representative of $\Phi(S)$ and $\theta'$ is a representative of $\Phi(S')$.
Consider the groups $G$, $G'$ spanned by the $\varphi_i$ (resp. $\varphi'_i$).

Suppose $G$ is not abelian.
Then if $S$ and $S'$ are formally diffeomorphic, there is a formal diffeomorphism $\psi$ conjugating $\theta$ and $\theta'$.
This $\psi$ realizes a conjugacy between $G$ and $G'$ so by Cerveau-Moussu's rigidity theorem \cite{cm}, $\psi$ is convergent.
This in turn implies that $\theta$ and $\theta'$ are analytically conjugated, so that $S$ and $S'$ are analytically diffeomorphic.
Note that since the diffeomorphisms $\varphi_i$ are tangent to the identity, the group $G$ is abelian only if the $\varphi_i$ are flows of a same formal vector field \cite{loray_pseudogroupe}.

This argument also works for non-canonical foliations: suppose that $S$ and $S'$ are analytic neighborhoods conjugated by a formal diffeomorphism $\psi$.
Suppose that there is on $S$ two convergent foliations $\F$ and $\G$ with tangency index $k\geq 1$ and $1$-form $\omega$ with simple zeroes.
Suppose that $\psi$ sends $\F$ and $\G$ to convergent foliations $\F'$ and $\G'$.
Suppose finally that the group $G$ spanned by the diffeomorphisms composing the invariant $Inv(S,\F,\G)$ of theorem \ref{thm_classif_bif} is not abelian.
Then $\psi$ converges.

\subsection{About degenerate cases}

If the $1$-form $\omega$ doesn't have simple zeroes, we can still obtain a classification of neighborhoods in $\ms{V}(C,k,\omega)$ by the same method.
The problem is that in this case some non-trivial local invariants can arise.
For genus $g=2$ curves, the local situations which can be involved were classified in \cite{thom_boletim}.
These local classifications can then be used to obtain a classification of bifoliated neighborhoods of genus $2$ curves even in the degenerate cases (see \cite{thom_these}), which in turn could give a complete formal classification of neighborhoods of genus $2$ curves with trivial normal bundle.

\bibliography{mybib}{}
\bibliographystyle{acm}

\textsc{Univ Rennes, CNRS, IRMAR - UMR 6625, F-35000 Rennes, France}

\textit{Email:} olivier.thom@univ-rennes1.fr

\end{document}